\documentclass[11pt]{article}

\usepackage{fancyhdr}
\usepackage{epsfig}
\usepackage{color}

\usepackage{graphicx}
\usepackage{latexsym}
\usepackage{amssymb}
\usepackage{amsfonts}
\usepackage{amsmath}
\usepackage{amsthm}

\theoremstyle{plain}
\newtheorem{thm}{Theorem}

\DeclareMathOperator{\PP}{{\bf P}}
\DeclareMathOperator{\E}{{\bf E}}
\DeclareMathOperator{\Var}{{\bf Var}}

\newcommand{\iid}{\mbox{i.i.d.\frenchspacing}}
\newcommand{\al}{\mbox{al.\frenchspacing}}

\begin{document}

\LARGE
\begin{center}
\textbf {On the asymptotic normality of finite population $L$-statistics}\\[2.25\baselineskip]

\small
\text{Andrius {\v C}iginas}\\

\medskip

{\footnotesize
Vilnius University Institute of Mathematics and Informatics, LT-08663 Vilnius, Lithuania
}\\[2.25\baselineskip]
\end{center}

\small
\begin{abstract}

   We give sufficient conditions for the asymptotic normality of linear combinations of order statistics ($L$-statistics) in the case of simple random samples without replacement. In the first case, restrictions are imposed on the weights of $L$-statistics. The second case is on trimmed means, where we introduce a new finite population smoothness condition.  
   
\end{abstract}
\vskip 2mm

\normalsize


\vskip 2mm

\noindent\textbf{MSC classes:} 62E20

\section{Introduction and results}\label{s:1}

In the case of independent and identically distributed (\iid{}) observations, asymptotic normality of $L$-statistics under various conditions was shown by Chernoff \emph{et \al{}} \cite{ChGJ_1967}, Shorack \cite{Sh_1972}, Stigler \cite{S_1973, S_1974} and Mason \cite{M_1981}, among others. See also Serfling \cite[Chapter 8]{S_1980}. In the case of samples \emph{drawn without replacement}, there are only the few works on the asymptotic normality of $L$-statistics, e.g., the paper of Shao \cite{Sh_1994}, where $L$-statistics under complex sampling designs are considered, and the work of Chatterjee \cite{Ch_2011} on the case of sample quantile.

Let ${\cal X}=\{x_1,\dots, x_N\}$ denote measurements of the study variable $x$ of the population ${\cal U}=\{u_1,\dots, u_N\}$ of subjects or objects, i.e., a real function $f\colon {\cal U} \to \mathbb R$ assigns a fixed value for each element of the population ${\cal U}$. Let $\mathbb X=\{X_1, \dots, X_n\}$ be measurements of units of the simple random sample of size $n<N$ drawn without replacement from the population. The observations $X_1, \dots, X_n$ are identically distributed, but they are not independent. Let $X_{1:n} \leq \dots \leq X_{n:n}$ denote the order statistics of $\mathbb X$. Define the $L$-statistic 
\begin{equation}\label{L_stat}
L_n=L_n(\mathbb X)=\frac{1}{n}\sum_{j=1}^n c_j X_{j:n}.
\end{equation}
Here $c_1,\ldots ,c_n$ is a given sequence of real numbers called weights. 
Usually these weights are determined by the weight function $J\colon(0,1) \to \mathbb R$ as follows:
\begin{equation*}\label{J_gen_c}
c_j=J\left(\frac{j}{n+1}\right), \quad 1\leq j\leq n.
\end{equation*} 
Further, when we talk about the asymptotics of $L$-statistics, we use centered statistics (\ref{L_stat}) with $n^{1/2}$ norming, i.e.,
\begin{equation}\label{S_stat}
S_n=S_n(\mathbb X)=n^{1/2}(L_n-\E L_n).
\end{equation} 
Denote $\tilde{\sigma}_{n}^2=\Var S_n$. We are interested in the normal approximation to the distribution function
\begin{equation*}\label{F_n}
F_n(x)=\PP \left\{ S_n\leq x\tilde{\sigma}_{n} \right\}. 
\end{equation*}
Note that for correct formulations of the following asymptotic results for finite population statistics, we need to consider a sequence of populations ${\cal X}_r=\{x_{r, 1}, \ldots, x_{r, N_r} \}$, with $N_r\to \infty$ as $r\to \infty$, and a sequence of statistics $L_{n_r}(\mathbb X_r)$, based on simple random samples $\mathbb X_r=\{X_{r, 1}, \ldots, X_{r, n_r}\}$ drawn without replacement from ${\cal X}_r$. In order to keep the notation simple, we shall skip the subscript $r$ in what follows.

The sample mean is the separate case of \eqref{L_stat}, where $c_j\equiv 1$, $1\leq j\leq n$. In this case, for samples drawn without replacement, the classical result on asymptotic normality was established by Erd\H{o}s and R\'enyi \cite{ER_1959}, see also H{\'a}jek \cite{H_1960}.
Similarly as in the case of \iid{} observations, the key asymptotic condition in \cite{ER_1959} is the Lindeberg-type condition: for every $\varepsilon>0$,
\begin{equation}\label{lind_0}
\sigma^{-2} \E (X_1-\E X_1)^2 \mathbb{I}{\{ \left| X_1-\E X_1 \right|>\varepsilon\tau\sigma\}}=o(1)
\quad \text{as} \quad N, n \to \infty,
\end{equation}
where $\sigma^2=\Var X_1$ and $\tau^2=Npq$ with $p=n/N$, $q=1-p$, and $\mathbb I\{ \cdot\}$ is the indicator function. Condition \eqref{lind_0} is called the Erd\H{o}s--R\'enyi condition.
Since $L$-statistics can be viewed as a certain generalization of the sample mean, one can expect that conditions, sufficient for the asymptotic normality, should be similar to that used in \cite{ER_1959}, but with some additional restrictions to the weights $c_1,\ldots ,c_n$.

On the other hand, $L$-statistics is a subclass of the more general class of symmetric statistics (symmetric functions of observations). An asymptotic behaviour of symmetric statistics differs not so much from that of the simplest linear statistic (the sample mean is an example), in the sense that, e.g., using Hoeffding's decomposition of Bloznelis and G{\" o}tze \cite{BG_2001}, we can write
\begin{equation}\label{short_exp}
S_n=U_1+R_1, \quad \text{where} \quad U_1=\sum_{i=1}^n g_1(X_i)
\end{equation}
is a linear statistic and (we expect that) $R_1$ is a stochastically smaller statistic. Then $S_n$ in (\ref{short_exp}) is asymptotically standard normal if its linear part $U_1$ is asymptotically standard normal, and $R_1$ is a degenerate statistic as the sample size 
\begin{equation*}
n_{*}:=\min\{n,N-n\}
\end{equation*}
increases. In particular, by \cite{BG_2001}, the components $U_1$ and $R_1$ are centered and uncorrelated, and (by Theorem 1 of \cite{BG_2001}) the variance of $R_1$ is bounded as follows: $\E R_1^2 \leq \delta_2$, where it is expected that the particular quantity $\delta_2=o(1)$ as $n_* \to \infty$. In the present paper, we apply the general result on asymptotic normality of the symmetric statistics (see Proposition 3 of \cite{BG_2001}) to the case of the $L$-statistics, i.e., we replace the condition imposed on $\delta_2$ by conditions expressed in terms of the weights $c_1, \ldots, c_n$ and the population ${\cal X}$.

We assume, without loss of generality, that the values of the population ${\cal X}$ are arranged in non-decreasing order, i.e., $x_1\leq \cdots \leq x_N$. Let us use the convention ${a\choose b}=0$ for $a<b$. In the case of $L$-statistic \eqref{S_stat}, the function $g_1(\cdot)$ in \eqref{short_exp} is represented by, for $1\leq k \leq N$,
\begin{equation}\label{g_1}
g_1(x_k)=-n^{-1/2}\sum_{j=1}^{n} c_j \sum_{i=1}^{N-1} \left ( \mathbb{I}{\{i\geq k\}}-\frac{i}{N}\right) {i-1\choose j-1}{N-i-1\choose n-j}{N-2\choose n-1}^{-1} (x_{i+1}-x_i),
\end{equation} 
see \v Ciginas \cite{C_2012}. Denote $\sigma_1^2=\E g_1^2(X_1)$.

First, we consider an $L$-statistic of the general form \eqref{L_stat}, and we will require a certain smoothness of its weight function $J(\cdot)$. Reformulate Erd\H{o}s--R\'enyi condition \eqref{lind_0}: for every $\varepsilon>0$,
\begin{equation}\label{lind_2}
 \sigma_1^{-2} \E g_1^2(X_1) \mathbb{I}{\{ \left| g_1(X_1) \right|>\varepsilon\tau\sigma_1\}}=o(1)
\quad \text{as} \quad n_* \to \infty.
\end{equation}
Then we have the following statement.
\begin{thm}\label{t:norm_l}
Assume that $n_* \to \infty$ and $\tilde{\sigma}_{n}\geq c_1>0$ for all $n_*$.
Suppose that $\E X_1^2\leq c_2<\infty$ and that $J(\cdot)$ is bounded and satisfies the H{\" o}lder condition of order $\delta>1/2$ on $(0, 1)$. Let \eqref{lind_2} hold.
Then $\tilde{\sigma}_{n}^{-1}S_n$ is asymptotically standard normal.
\end{thm} 
\noindent Note that, in comparison to the case of the sample mean, conditions on the finite population ${\cal X}$ remain very mild. Assumptions of Theorem \ref{t:norm_l}, sufficient for the asymptotic normality of $L$-statistics, are similar to that obtained by Stigler \cite{S_1974} in the \iid{} case.

Second, we consider an important special case of \eqref{L_stat}, i.e., the trimmed means. The trimmed mean is defined as follows: for any fixed numbers $0<t_1<t_2<1$,
\begin{equation*}\label{trimm_mean}
M_{t_1;t_2}= ([t_2 n] - [t_1 n])^{-1} \sum_{j=[t_1 n]+1}^{[t_2 n]} X_{j:n},
\end{equation*}
where $[\cdot]$ is the greatest integer function. The statistic $M_{t_1;t_2}$ is represented by the weight function $J(u)=(t_2-t_1)^{-1}\mathbb I\{t_1<u<t_2\}$. This function is not sufficiently smooth, i.e., $J(u)$ is bounded, but it does not satisfy the H{\" o}lder condition. Let us introduce an additional smoothness condition for the population ${\cal X}$. Assume that, without loss of generality, $x_1\leq \cdots \leq x_N$. Suppose that, for some constants $C>0$ and $1/2 <\delta \leq 1$, the inequality
\begin{equation}\label{ac_cond}
\left| x_m-x_l \right| \leq C N^{-\delta} \left| m-l \right|
\end{equation}
is satisfied for all $1\leq l<m \leq N$.
\begin{thm}\label{t:norm_trim}
Assume that $n_* \to \infty$ and $\tilde{\sigma}_{n}\geq c_1>0$ for all $n_*$.
Say that $\E X_1^2\leq c_2<\infty$. Assume that \eqref{ac_cond} is satisfied for some $1/2 <\delta \leq 1$, and $(1-n/N)^{-1} \, n^{1/2} N^{\delta-1} \to \infty$. Then, in the case of a trimmed mean, $\tilde{\sigma}_{n}^{-1}S_n$ is asymptotically standard normal.
\end{thm}
\noindent In the case of \iid{} observations, it was shown by Stigler \cite{S_1973} that in order for the trimmed mean to be asymptotically normal, it is necessary and sufficient that the sample is trimmed at sample quantiles for which the corresponding population quantiles are uniquely defined. Thus, the conditions of Theorem \ref{t:norm_trim} seem too strong. On the other hand, in finite population settings, the new smoothness condition (\ref{ac_cond}) has a specific interpretation. Let us take $l=1$ and $m=N$. If the population ${\cal X}$ is bounded, then the condition is satisfied for $\delta=1$. For any finite population, condition (\ref{ac_cond}) is satisfied in the marginal case of $\delta=1/2$. The latter fact follows from the Nair--Thomson inequality $\left|x_N-x_1\right| \leq \sigma\sqrt{2N}$ (see, e.g., Balakrishnan \emph{et \al} \cite{BCP_2003}). Thus, condition (\ref{ac_cond}) seems very mild for small $\theta>0$ in $\delta=1/2+\theta$, i.e., it holds for most of possible populations. Obviously, if we are interested in the asymptotic normality of the trimmed means, then, by the conditions of Theorem \ref{t:norm_trim}, for small $\theta$ we should have $n \to \infty$ quite quickly as $N \to \infty$, while in the case of $\delta=1$ it suffices that $n \to \infty$ arbitrarily slowly with respect to the grow of the population size $N$.

\section{Proofs}\label{s:2}

In the proofs of Theorems \ref{t:norm_l} and \ref{t:norm_trim}, we assume that, without loss of generality, $x_1\leq \cdots \leq x_N$.

\begin{proof}[Proof of Theorem \ref{t:norm_l}]

First, we show that $\tilde{\sigma}_{n}$ is bounded as $n_* \to \infty$. Then the condition: for every $\varepsilon>0$,
\begin{equation}\label{lind_1}
n_*\E g_1^2(X_1) \mathbb{I}{\{ g_1^2(X_1)>\varepsilon\}}=o(1)
\quad \text{as} \quad n_* \to \infty
\end{equation}
of Proposition 3 in Bloznelis and G{\" o}tze \cite{BG_2001} is equivalent (see ibidem) to condition \eqref{lind_2}. It is shown by \v Ciginas and Pumputis \cite{CP_2012} that, for any symmetric statistic, the inequality
\begin{equation}\label{like_extr}
\tilde{\sigma}_{n}^2\leq \frac{1}{2} n\left(1-\frac{n}{N}\right) \E \left(\mathbb D_1 S_n \right)^2
\end{equation}
holds. Here $\mathbb D_1 S_n=S_n(\mathbb X_1\backslash{\{X_{n+1}\}})-S_n(\mathbb X_1\backslash{\{X_1\}})$, where $\mathbb X_1=\{X_1, \dots, X_{n+1}\}$ is the extended sample. Introduce the events $\mathfrak{R}_{1;ij}=\{ R_{1:2}=i, R_{n+1:2}=j\}$, $1\leq i<j\leq n+1$, where $R_{1:2}<R_{n+1:2}$ denote the order statistics of the ranks $\{R_1, R_{n+1}\}$ of $\{X_1, X_{n+1}\}$ in the set $\mathbb X_1$. Here all ranks $\{R_1,\ldots,R_{n+1}\}$ of $\mathbb X_1$ are distinct if, in the case of ties on ${\cal X}$, we order (select ranks for) tied observations randomly with equal probabilities. The probabilities of the events are
\begin{equation*}\label{p1ij}
p_{1;ij}:=\PP \left\{ \mathfrak{R}_{1;ij}\right\}={n+1\choose 2}^{-1}. 
\end{equation*}
Since $J(\cdot)$ is bounded, there exists an absolute constant $a$ that
\begin{equation}\label{smooth_0}
\max_{1\leq p \leq n} \left| c_p \right| \leq a
\end{equation}
for all $n$. By Lemma 2 of \cite{C_2012} and (\ref{smooth_0}), we obtain
\begin{equation}\label{med_ineq}
\begin{split} 
\E \left(\mathbb D_1 S_n \right)^2&=\sum_{1\leq i<j\leq n+1} \E \left[ \left(\mathbb D_1 S_n \right)^2 \,\middle|\, \mathfrak{R}_{1;ij}\right] p_{1;ij} \\
&\leq n^{-1} \sum_{1\leq i<j\leq n+1} \E \left[ \bigg( \sum_{p=R_{1:2}}^{R_{n+1:2}-1} c_p \mathbf\Delta_{p:n+1} \bigg)^2 \,\middle|\, \mathfrak{R}_{1;ij}\right] p_{1;ij} \\
&\leq a^2 n^{-1} \sum_{1\leq i<j\leq n+1} \E \left[ \left( X_{j:n+1}-X_{i:n+1} \right)^2 \,\middle|\, \mathfrak{R}_{1;ij}\right] p_{1;ij},
\end{split}
\end{equation}
where $\mathbf\Delta_{p:n+1}=X_{p+1:n+1}-X_{p:n+1}$, $1\leq p\leq n$ denote the spacings of the sample $\mathbb X_1$.
Since the events $\mathfrak{R}_{1;ij}$ and $\mathfrak{B}_{1;ijlm}=\left\{ X_{i:n+1}=x_l, X_{j:n+1}=x_m  \right\}$, $1\leq l<m\leq N$ are independent, for $x_1<\cdots <x_N$ we get
\begin{equation*}
p_{1;ijlm}:=\PP \left\{ \mathfrak{B}_{1;ijlm} \,\middle|\, \mathfrak{R}_{1;ij} \right\}=
{l-1\choose i-1}{m-l-1\choose j-i-1}{N-m\choose n+1-j}\left/{N\choose n+1}\right. .
\end{equation*}
For $x_1\leq \cdots \leq x_N$ these probabilities are the same. It follows from an argument similar to Lemma 2.1 of Balakrishnan \emph{et \al} \cite{BCP_2003}. We also have that, by the generalized Vandermonde identity,
\begin{equation*}
\begin{split}
\sum_{1\leq i<j\leq n+1} p_{1;ijlm}
&={N\choose n+1}^{-1} \sum_{s=0}^{n-1} \sum_{t=0}^{n-1-s} {l-1\choose s}{m-l-1\choose t}{N-m\choose n-1-s-t} \\
&={N\choose n+1}^{-1} {N-2\choose n-1}.
\end{split}
\end{equation*} 
Then note that 
\begin{equation*}\label{var_expr}
\Var X_1=\frac{1}{N^2}\sum_{1\leq l<m\leq N} (x_m-x_l)^2
\end{equation*}
and continue (\ref{med_ineq}):
\begin{equation*}\label{end_ineq}
\begin{split}
\E \left(\mathbb D_1 S_n \right)^2&
\leq  a^2 n^{-1} \sum_{1\leq i<j\leq n+1} \bigg[ \sum_{1\leq l<m\leq N} (x_m-x_l)^2 p_{1;ijlm} \bigg] p_{1;ij} \\
&= a^2 n^{-1} {n+1\choose 2}^{-1} {N\choose n+1}^{-1} {N-2\choose n-1} \sum_{1\leq l<m\leq N} (x_m-x_l)^2 
= 2 a^2 n^{-1} \frac{N}{N-1} \Var X_1.
\end{split}
\end{equation*}
Finally, from (\ref{like_extr}) we get
\begin{equation*}
\tilde{\sigma}_{n}^2\leq a^2 \frac{N-n}{N-1}  \Var X_1 =O(1) \quad \text{as} \quad n_* \to \infty.
\end{equation*}

Second, we show that, under the conditions of the theorem, the condition $\delta_2(S_n)=o(1)$ as $n_* \to \infty$ of Proposition 3 in \cite{BG_2001} is satisfied. Here $\delta_2(S_n)=\E \left(n_*{\mathbb D}_2 S_n\right)^2$, where
\begin{equation*}
{\mathbb D}_2 S_n=S_n(\mathbb X_2\backslash{\{X_{n+1},X_{n+2}\}})-S_n(\mathbb X_2\backslash{\{X_1,X_{n+2}\}})-S_n(\mathbb X_2\backslash{\{X_2,X_{n+1}\}})+S_n(\mathbb X_2\backslash{\{X_1,X_2\}})
\end{equation*}
with the extended sample $\mathbb X_2=\{X_1, \dots, X_{n+2}\}$, see \cite{BG_2001}. Similarly, introduce the events $\mathfrak{R}_{2;ij}=\{ R_{2:4}=i, R_{n+1:4}=j\}$, $1\leq i<j\leq n+2$, where $R_{1:4}<R_{2:4}<R_{n+1:4}<R_{n+2:4}$ denote the order statistics of the ranks $\{R_1, R_2, R_{n+1}, R_{n+2}\}$ of $\{X_1, X_2, X_{n+1}, X_{n+2}\}$ in the set $\mathbb X_2$. Now
\begin{equation}\label{p2ij}
p_{2;ij}:=\PP \left\{ \mathfrak{R}_{2;ij} \right\}=
{i-1\choose 1}{n+2-j\choose 1}\left/{n+2\choose 4}\right. .
\end{equation}
We also similarly have 
\begin{equation*}
p_{2;ijlm}:=\PP \left\{ \mathfrak{B}_{2;ijlm} \,\middle|\, \mathfrak{R}_{2;ij} \right\}=
{l-1\choose i-1}{m-l-1\choose j-i-1}{N-m\choose n+2-j}\left/{N\choose n+2}\right. ,
\end{equation*}
where the events $\mathfrak{R}_{2;ij}$ and $\mathfrak{B}_{2;ijlm}=\left\{ X_{i:n+2}=x_l, X_{j:n+2}=x_m  \right\}$, $1\leq l<m\leq N$ are independent.
Since $J(\cdot)$ satisfies the H{\" o}lder condition of order $\delta>1/2$ on $(0, 1)$, we find that
\begin{equation*}
\left| c_p-c_{p-1} \right| = \left| J\left(\frac{p}{n+1}\right)-J\left(\frac{p-1}{n+1}\right) \right| \leq B (n+1)^{-\delta}
\end{equation*}
or 
\begin{equation}\label{smooth_1}
\max_{2\leq p \leq n} \left| c_p-c_{p-1} \right| \leq B (n+1)^{-\delta}, \quad \text{for some} \quad \delta>1/2.
\end{equation}
By Lemma 2 of \cite{C_2012} and (\ref{smooth_1}), we obtain
\begin{equation}\label{med_ineq_2}
\begin{split}
\delta_2(S_n)&=n_*^2 \sum_{1\leq i<j\leq n+2} \E \left[ \left(\mathbb D_2 S_n \right)^2 \,\middle|\, \mathfrak{R}_{2;ij}\right] p_{2;ij} \\
&\leq n_*^2 n^{-1} \sum_{1\leq i<j\leq n+2} \E \left[ \bigg( \sum_{p=R_{2:4}}^{R_{n+1:4}-1} (c_p-c_{p-1}) \mathbf\Delta_{p:n+2} \bigg)^2 \,\middle|\, \mathfrak{R}_{2;ij}\right] p_{2;ij} \\
&\leq B^2 n_*^2 n^{-1}(n+1)^{-2\delta} \sum_{1\leq i<j\leq n+2} \E \left[ \left( X_{j:n+2}-X_{i:n+2} \right)^2 \,\middle|\, \mathfrak{R}_{2;ij}\right] p_{2;ij} \\
&= B^2 n^{1-2\delta} \sum_{1 \leq l<m \leq N} \lambda_{2;lm} (x_m-x_l)^2,
\end{split}
\end{equation}
where $\lambda_{2;lm}=\sum_{1 \leq i<j \leq n+2} p_{2;ij}p_{2;ijlm}$. Taking $j=i+1$ and applying $\max_{0\leq u \leq 1} u(1-u)\leq 1/4$, for all $1\leq i<j\leq n+2$, we get the inequalities
\begin{equation*}
p_{2;ij} \leq n^2 {n+2\choose 4}^{-1} \frac{i-1}{n} \left( 1-\frac{i-1}{n} \right) 
\leq \frac{1}{4} n^2 {n+2\choose 4}^{-1}.
\end{equation*}
Then, noting that, by the generalized Vandermonde identity, 
\begin{equation*}
\sum_{1\leq i<j\leq n+2} p_{2;ijlm}={N\choose n+2}^{-1} {N-2\choose n},
\end{equation*}
we obtain, for all $1 \leq l<m \leq N$,
\begin{equation*}
\lambda_{2;lm} \leq \frac{1}{4} n^2 {n+2\choose 4}^{-1} {N\choose n+2}^{-1} {N-2\choose n} \leq 24 N^{-2}.
\end{equation*}
Finally, it follows from this bound and (\ref{med_ineq_2}) that
\begin{equation*}
\delta_2(S_n) \leq 24 B^2 n^{1-2\delta} \Var X_1 = o(1) \quad \text{as} \quad n_* \to \infty.
\end{equation*}
All the conditions of Proposition 3 in \cite{BG_2001} are verified. Thus, the theorem is proven.
\end{proof}

\begin{proof}[Proof of Theorem \ref{t:norm_trim}]

First, we show that condition \eqref{ac_cond} with $(1-n/N)^{-1} \, n^{1/2} N^{\delta-1} \to \infty$ imply \eqref{lind_1}. Noting that
\begin{equation*}
\sum_{j=1}^{n} {i-1\choose j-1}{N-i-1\choose n-j}{N-2\choose n-1}^{-1}=1,
\end{equation*}
applying \eqref{smooth_0} and \eqref{ac_cond}, we get from \eqref{g_1}, 
\begin{equation*}
\max_{1\leq k\leq N} |g_1(x_k)|\leq aC n^{-1/2} N^{-\delta} \max_{1\leq k\leq N} \sum_{i=1}^{N-1} \left| \mathbb{I}{\{i\geq k\}}-\frac{i}{N} \right|=\frac{aC}{2} n^{-1/2} N^{-\delta} (N-1).
\end{equation*}
Therefore, for a fixed $\varepsilon>0$, 
\begin{equation*}
\mathbb{I}{\{ |g_1(X_1)|>\varepsilon\}} \leq \mathbb{I}{\Big\{ \max_{1\leq k\leq N} |g_1(x_k)|>\varepsilon\Big\}} \leq \mathbb{I}{\Big\{ \frac{aC}{2} n^{-1/2} N^{1-\delta}>\varepsilon\Big\}}.  
\end{equation*}
We obtain from here and from (3.9) of \cite{C_2012} that
\begin{equation*}
n_*\E g_1^2(X_1) \mathbb{I}{\{ |g_1(X_1)|>\varepsilon\}}\leq \mathbb{I}{\Big\{ \frac{aC}{2} n^{-1/2} N^{1-\delta}>\varepsilon\Big\}} n_*\E g_1^2(X_1) 
\leq 4a^2 \mathbb{I}{\Big\{ n^{-1/2} N^{1-\delta}>\frac{2\varepsilon}{aC}\Big\}} \E |X_1|^2.
\end{equation*} 
Condition \eqref{lind_1} is proven.

Second, as in the proof of Theorem \ref{t:norm_l}, we verify the condition $\delta_2(S_n)=o(1)$ as $n_* \to \infty$. Write, for short, $s=[t_1 n]+1$ and $t=[t_2 n]$. Similarly, applying Lemma 2 of \cite{C_2012}, we obtain 
\begin{equation}\label{med_ineq_3}
\delta_2(S_n) \leq \frac{n_*^2 n}{(t-s+1)^2} \sum_{1\leq i<j\leq n+2} p_{2;ij} \E A_{ij}^2(s,t),
\end{equation}
where $p_{2;ij}$ is given by (\ref{p2ij}) and
\begin{equation*}
A_{ij}(s,t)=\sum_{p=i}^{j-1} (\tilde{c}_p-\tilde{c}_{p-1}) \mathbf{\Delta}_{p:n+2} \quad \text{with} \quad \tilde{c}_p=\mathbb I\{s\leq p \leq t\}.
\end{equation*} 
We can assume, without loss of generality, that $n>(t_2-t_1)^{-1}$. Then we have $s<t$. It also follows from the inequality $[t_2 n]-[t_1 n] \geq t_2 n - 1 - t_1 n$ and from the same assumption that, for some constant $C_1>0$,
\begin{equation}\label{floor_ineq}
\frac{n^2}{(t-s+1)^2} \leq \left( t_2-t_1-\frac{1}{n} \right)^{-2} \leq C_1.
\end{equation}
Let us decompose ${\cal I}=\{ (i, j): 2\leq i<j\leq n+1\}$, for fixed $s<t$, into mutually disjoint subsets
\begin{align*}
&{\cal I}_1=\{ (i, j): t+2\leq i<j\leq n+1 \}, \\
&{\cal I}_2=\{ (i, j): 2\leq i<j\leq s \}, \\
&{\cal I}_3=\{ (i, j): s+1\leq i<j\leq t+1 \}, \\
&{\cal I}_4=\{ (i, j): s+1\leq i\leq t+1, \, t+2\leq j\leq n+1 \}, \\
&{\cal I}_5=\{ (i, j): 2\leq i\leq s, \, s+1\leq j\leq t+1 \}, \\
&{\cal I}_6=\{ (i, j): 2\leq i\leq s, \, t+2\leq j\leq n+1 \},
\end{align*}
such that ${\cal I}={\cal I}_1 \cup \cdots \cup {\cal I}_6$. Then we get
\begin{equation*}
A_{ij}(s,t)=
\begin{cases}
0                                      &\text{if $(i, j)\in {\cal I}_1 \cup {\cal I}_2 \cup {\cal I}_3$,} \\
-\tilde{c}_t \mathbf{\Delta}_{t+1:n+2} &\text{if $(i, j)\in {\cal I}_4$,} \\
\tilde{c}_s \mathbf{\Delta}_{s:n+2} &\text{if $(i, j)\in {\cal I}_5$,} \\
\tilde{c}_s \mathbf{\Delta}_{s:n+2}-\tilde{c}_t \mathbf{\Delta}_{t+1:n+2} &\text{if $(i, j)\in {\cal I}_6$.} 
\end{cases}
\end{equation*}
Now, by collecting the terms of the sum $\sum_{i<j}$ with the same value of $\E A_{ij}^2(s,t)$ in (\ref{med_ineq_3}), applying $\E (\mathbf{\Delta}_{t+1:n+2} - \mathbf{\Delta}_{s:n+2})^2 \leq \E \mathbf{\Delta}_{t+1:n+2}^2 + \E \mathbf{\Delta}_{s:n+2}^2$, and then collecting terms with $\E \mathbf{\Delta}_{t+1:n+2}^2$ and $\E \mathbf{\Delta}_{s:n+2}^2$, and also invoking inequality (\ref{floor_ineq}), we obtain 
\begin{equation}\label{med_ineq_4}
\begin{split}
\delta_2(S_n) \leq C_1 n_*^2 n^{-1} {n+2\choose 4}^{-1}
&\bigg[ {t+1\choose 2}{n-t+1\choose 2} \E \mathbf{\Delta}_{t+1:n+2}^2 \\
& + {s\choose 2}\bigg\{ (n-t+1)^2-{n-s+1\choose 2} \bigg\} \E \mathbf{\Delta}_{s:n+2}^2 \bigg].
\end{split}
\end{equation}
By applying the simple inequality ${u\choose v} \leq u^v/v!$, we derive
\begin{equation}\label{bin_ineq_1}
{t+1\choose 2}{n-t+1\choose 2} \leq \frac{(n+2)^4}{4} 
\left[  \frac{t+1}{n+2} \left( 1-\frac{t+1}{n+2} \right) \right]^2 \leq \frac{(n+2)^4}{64}. 
\end{equation}
Taking $s=t$, very similarly we get
\begin{equation}\label{bin_ineq_2}
{s\choose 2}(n-t+1)^2 \leq \frac{(n+1)^4}{2} 
\left[  \frac{t}{n+1} \left( 1-\frac{t}{n+1} \right) \right]^2 \leq \frac{(n+1)^4}{32}. 
\end{equation}
Next, it is easy to calculate (invoking Lemma 2.1 of Balakrishnan \emph{et \al} \cite{BCP_2003}) that, for $1\leq p\leq n+1$,
\begin{equation*}
\E \mathbf{\Delta}_{p:n+2}^2=
{N\choose n+2}^{-1} \sum_{1\leq l<m\leq N} {l-1\choose p-1}{m-l-1\choose 0}{N-m\choose n+1-p} (x_m-x_l)^2.
\end{equation*}
Then, using (\ref{ac_cond}), we obtain
\begin{equation}\label{ac_cond_conseq}
\E \mathbf{\Delta}_{p:n+2}^2 \leq
\frac{C^2}{N^{2\delta}} {N\choose n+2}^{-1} \sum_{1\leq l<m\leq N} (m-l)^2 {l-1\choose p-1}{N-m\choose n+1-p} 
=\frac{C^2}{N^{2\delta}} \frac{(N+1)(2N-n)}{(n+3)(n+4)}.
\end{equation}
Here the last equality is obtained by applying simple binomial identities ${m\choose k}=\frac{m}{k}{m-1\choose k-1}$, $1\leq k\leq m$,
\begin{equation*}
\sum_{j=k}^m {j\choose k}={m+1\choose k+1}, \quad 0\leq k\leq m \quad\,  \text{and} \quad\, \sum_{p=0}^m {p\choose j}{m-p\choose k-j}={m+1\choose k+1}, \quad 0\leq j\leq k\leq m.
\end{equation*}
Indeed, for instance,
\begin{equation*}
\begin{split}
\sum_{1\leq l<m\leq N} & m^2  {l-1\choose p-1}{N-m\choose n+1-p}=
\sum_{m=2}^N \bigg\{ \sum_{l=1}^{m-1} {l-1\choose p-1} \bigg\} m^2 {N-m\choose n+1-p} \\ 
&=\sum_{m=2}^N m^2 {m-1\choose p}{N-m\choose n+1-p}=(p+1)\sum_{m=2}^N m {m\choose p+1}{N-m\choose n+1-p} \\
&=(p+1)(p+2)\sum_{m=2}^N {m+1\choose p+2}{N-m\choose n+1-p}-(p+1)\sum_{m=2}^N {m\choose p+1}{N-m\choose n+1-p} \\
&=(p+1)(p+2) {N+2\choose n+4} - (p+1) {N+1\choose n+3},
\end{split}
\end{equation*}
and so on. Finally, applying (\ref{bin_ineq_1}), (\ref{bin_ineq_2}) and (\ref{ac_cond_conseq}), and $n_* \leq 2n(1-n/N)$, we continue (\ref{med_ineq_4}):  
\begin{equation*}
\begin{split}
\delta_2(S_n) &\leq C_1 n_*^2 n^{-1} {n+2\choose 4}^{-1}
\frac{C^2}{N^{2\delta}} \frac{(N+1)(2N-n)}{(n+3)(n+4)}
\left[ \frac{(n+2)^4}{64} + \frac{(n+1)^4}{32} \right] \\
&\leq C_2 \left( 1-\frac{n}{N} \right)^2 \frac{N^{2(1-\delta)}}{n}= o(1) \quad \text{as} \quad n_* \to \infty,
\end{split}
\end{equation*}
for some constant $C_2>0$. The theorem is proven.
\end{proof}

\small

\bibliographystyle{abbrv} 
\bibliography{liter_abc}
 
\end{document}